\documentclass[12pt]{article}

\usepackage{authblk}

\usepackage{amsfonts}
\usepackage{graphicx}
\usepackage{tabularx}
\usepackage{array}
\usepackage[usenames,dvipsnames]{color}
\usepackage{comment}
\usepackage{amsmath}
\usepackage{amsthm}
\usepackage{amssymb}
\usepackage{fullpage}
\usepackage[dvipsnames]{xcolor}
\usepackage{tikz}
\usepackage{listings}
\usepackage{dsfont}
\usepackage{enumerate}
\usepackage{romannum}

\newtheorem{theorem}{Theorem}[section]

\newtheorem{conjecture}[theorem]{Conjecture}
\newtheorem{problem}[theorem]{Problem}

\newtheorem{lemma}[theorem]{Lemma}

\theoremstyle{definition}

\DeclareMathOperator{\degb}{degb}
\DeclareMathOperator{\degr}{degr}

\def\epsilon{\varepsilon}

\title{A Note on Large Degenerate Induced Subgraphs 
\\ in Sparse Graphs}
\author[1]{Alexander Clow\,\thanks{Supported by the Natural Sciences and Engineering Research Council of Canada (NSERC) through PGS D-601066-2025}}
\author[2]{Sean Kim\,}
\author[3]{Ladislav Stacho\,\thanks{Supported by  the Natural Sciences and Engineering Research Council of Canada (NSERC) through Discovery Grant R611368.}}
\date{}

\affil[1,2,3]{ \small{Department of Mathematics, Simon Fraser University}}

\begin{document}
\pagenumbering{arabic}

\maketitle

\begin{abstract}
Given a graph $G$ and a non-negative integer $d$ let $\alpha_d(G)$ be the order of a largest induced $d$-degenerate subgraph of $G$. We prove that for any pair of non-negative integers $k>d$, if $G$ is a $k$-degenerate graph, then $\alpha_d(G) \geq \max\{ \frac{(d+1)n}{k+d+1}, n - \alpha_{k-d-1}(G)\}$. For $k$-degenerate graphs this improves a more general lower bound of Alon, Kahn, and Seymour. By modifying our argument we obtain improved lower bound on $\alpha_d(G)$ for graphs of bounded genus. This extends earlier work on degenerate subgraphs of planar graphs.
\end{abstract}

\section{Introduction} 

A graph $G$ is $d$-degenerate if every subgraph $H$ of $G$ has minimum degree at most $d$.
The degeneracy of a graph $G$ is the least $d$ such that $G$ is $d$-degenerate.
Thus, a graph is $0$-degenerate  if and only if it has no edges, and $1$-degenerate if and only if it is a forests.
Given a fixed integer $d$ and a graph $G$,
we let $\alpha_d(G)$ be the order of a largest induced $d$-degenerate subgraph of $G$.
Hence, by setting $d=0$ one can recover the independence number of the graph $G$.
For more background and definitions in graph theory we refer the reader to \cite{west2001introduction}.

Alon, Kahn, and Seymour 
\cite{alon1987large}
began the study of $\alpha_d(G)$
when they
proved that in any graph $G$,
\[
\alpha_d(G) \geq \sum_{v\in V(G)} \min\Big\{1, \frac{d+1}{\deg(v)+1} \Big\}.
\]
Notice that this bound is tight for any disjoint union of cliques.
As remarked in \cite{alon1987large}, when $d = 0$ this returns a classic lower bound for independence number
proven by Caro \cite{caro1979new} and independently by Wei \cite{wei1981lower}.
As a corollary of this result, 
graphs with average degree $D \geq 2d$ have
\[
\alpha_d(G) \geq \frac{(d+1)n}{D+1}.
\]
Observe that every $k$-degenerate graph has average degree at most $2k$, with equality for any maximally $k$-degenerate graph.
So for any $d<k$, Alon, Kahn, and Seymour proved that in any $k$-degenerate graph
$\alpha_d(G) \geq \frac{(d+1)n}{2k+1}$.

There is a long history of studying $\alpha_d$ for planar graphs.
As corollary of the Four Colour Theorem, for all planar graphs $G$,  $\alpha_0(G) \geq \frac{n}{4}$.
This bound is tight as witnessed by $K_4$.
Proving this lower bound without the Four Colour Theorem remains open, although there has been some progress, see \cite{albertson1976lower, cranston2016planar}.

Similarly, $K_4$ demonstrates that there are planar graphs with $\alpha_1(G) = \frac{n}{2}$.
Albertson and Berman \cite{albertson1979conjecture} and Akiyama and Watanabe \cite{akiyama1987maximum} independently conjectured 
that $\alpha_1(G) \geq \frac{n}{2}$ for all planar graphs.
This conjecture has seen some significant attention since being proposed.
Indeed,  a corollary of Borodin's proof \cite{borodin1979acyclic} that all planar graphs are acyclically $5$-colourable is that for all planar graphs
$G$, $\alpha_1(G) \geq \frac{2n}{5}$.
This remains the best general lower bound on $\alpha_1$ for planar graphs.
The conjecture is known to be true for triangle-free planar graphs, see \cite{salavatipour2006large}.
The current best lower bound on $\alpha_1$ for triangle-free planar graphs $G$
is from Dross, Montassier, and Pinlou \cite{dross2019large}
who proved that $\alpha_1(G) \geq \frac{6n+7}{11}$.

More recently, 
studying $\alpha_d$ for $2 \leq d\leq 4$ in planar graphs has garnered some attention.
The octahedron witnesses that there exists planar graphs with $\alpha_2(G) = \frac{2n}{3}$.
Gu, Kierstead, Oum, Qi, and Zhu \cite{gu20223} conjecture that every planar graph has $\alpha_2(G) \geq \frac{2n}{3}$.
Notably, Dvo{\v{r}}{\'a}k and Kelly \cite{dvovrak2018induced} proved that if $G$ is a triangle-free planar graph
then the stronger bound $\alpha_2(G) \geq \frac{4n}{5}$ holds.
When $d = 3$
the octahedron and icosahedron witness that there exists planar graphs with $\alpha_3(G) = \frac{5n}{6}$.
Gu, Kierstead, Oum, Qi, and Zhu \cite{gu20223} showed that all planar graphs have $\alpha_3(G) \geq \frac{3n}{4}$.
For $d = 4$
the icosahedron witnesses that there exists planar graphs with $\alpha_4(G) = \frac{11n}{12}$.
The best bound here is from 
Lukot'ka, Maz{\'a}k, and Zhu \cite{lukot2015maximum}
who showed that for all planar graphs $\alpha_4(G) \geq \frac{8n}{9}$.

In this paper, we consider 
this problem for $k$-degenerate graphs.
Thus, our results are for a more general class of graphs than planar graphs, but more specific than the average degree cases 
studied in \cite{alon1987large}. 
Our primary contribution is the following theorem.

\begin{theorem}\label{Thm: Main Degen}
    Let $k > d$ be non-negative integers.
    If $G$ is a $k$-degenerate graph, then
    \[
    \alpha_d(G) \geq \max \Big\{ \frac{(d+1)n}{k+d+1}, n - \alpha_{k-d-1}(G)\Big\}.
    \]
    Furthermore, $V(G)$ can be partitioned into sets $X$ and $Y$ such that $G[X]$ is $d$-degenerate
    and $G[Y]$ is $(k-d-1)$-degenerate.
\end{theorem}

Since Theorem~\ref{Thm: Main Degen} provides linear lower bound on $\alpha_d$, 
for pairs of integers $k>d$, we define
\[
\alpha_d(k) := \inf_{G \text{ with degeneracy $k$}} \frac{\alpha_d(G)}{ \vert V(G) \vert },
\]
and note that $\alpha_d(k)$ is bounded away from $0$.

Trivially, as witnessed by the clique $K_{k+1}$, for all $k>d$, $\alpha_d(k) \leq \frac{d+1}{k+1}$.
Together with Theorem~\ref{Thm: Main Degen} this implies $\frac{2}{k+2} \leq \alpha_1(k) \leq \frac{2}{k+1}$ for all $k$.
Meaning that for large $k$ our bound in Theorem~\ref{Thm: Main Degen} is close to optimal.
On the other hand as $d$ grows relative to $k$, our lower bound that $\alpha_d(G) \geq \frac{(d+1)n}{k+d+1}$ tends to 
Alon, Kahn, and Seymour's bound for graphs with average degree at most $2k$.
In these cases, the other term in our lower bound can be helpful.
For example, Theorem~\ref{Thm: Main Degen} implies that if $G$ is $k$-degenerate 
and the independence number of $G$ is $\frac{n}{k+1}$, the minimum for maximally $k$-degenerate graphs, 
then $\alpha_d(G) \geq \frac{(d-1)n}{k+1}$.
This nearly matches the $\alpha_d(k)$ upper bound witnessed by cliques.

Notice that for certain values of $k$ and $d$,
for example $k = 3$ and $d=1$, $\alpha_d(k) = \frac{d+1}{k+1}$,
while for other values of $k$ and $d$, for example $k=2$ and $d=1$, $\alpha_d(k) < \frac{d+1}{k+1}$.
To see that $\alpha_1(3) = \frac{1}{2}$ notice that Theorem~\ref{Thm: Main Degen} implies 
$\alpha_1(G) \geq n - \alpha_1(G)$ for any $3$-degenerate graph $G$.
Meanwhile, $K_4$ with one edge subdivided is an example of $2$-degenerate graph that witnesses $\alpha_1(2) \leq \frac{3}{5}< \frac{2}{3}$.

Given the literature on planar graphs, studying this problem for graphs of bounded genus is natural.
For a preliminary on graphs on more general surfaces than the plane we recommend \cite{MoharThomassen}.
By adapting our proof of Theorem~\ref{Thm: Main Degen} we obtain the following bound for graphs of bounded genus.

\begin{theorem}\label{Thm: Surfaces}
    Let $g$ be an integer. If $G$ is a graph with genus at most $g$, then 
    \begin{itemize}
        \item $\alpha_1(G) \geq \frac{2n-24g+2}{7}$, and
        \item $\alpha_2(G) \geq \frac{n-2g+4}{3}$, and
        \item $\alpha_3(G) \geq \frac{n}{2} -g+1$, and
        \item $\alpha_4(G) \geq \frac{3n - 2g + 4}{5}$, and
        \item $\alpha_5(G) \geq \frac{2n - g + 2}{3}$.
    \end{itemize}
\end{theorem}

Recalling that for planar graphs, the best lower bound on $\alpha_1$ comes from acyclic colouring,
we note that this approach an acyclic colouring argument does not improve Theorem~\ref{Thm: Surfaces} in the $d=1$ case for graphs with larger genus.
This is because there are graphs with genus $g$ and acyclic chromatic number $\Omega(g^{\frac{4}{7}-o(1)})$, see \cite{alon1996acyclic}.
Even if we restrict to locally planar graphs, which are the classic example of graphs of large genus that behave like planar graphs,
one does not achieve an improvement to Theorem~\ref{Thm: Surfaces}.
This is because the best upper bound for the acyclic chromatic number of locally planar graphs is $7$ \cite{kawarabayashi2010star}. 

The paper is organized as follows.
In Section~2 we prove Theorem~\ref{Thm: Main Degen}.
In Section~3 we prove Theorem~\ref{Thm: Surfaces}.
While Section~4 discusses our computational approach to find improved upper bounds for $\alpha_d(k)$
for small values of $k$ and $d$.
We conclude with a discussion of future work.

\section{Degenerate Graphs}

This section will be dedicated to proving Theorem~\ref{Thm: Main Degen}.
Before launching into the proof of Theorem~\ref{Thm: Main Degen},
we note that an equivalent condition for a graph to be $d$-degenerate 
is the existence of linear ordering of its vertices
where each vertex has at most $d$-neighbours that come later it in the ordering.
See \cite{lick1970k}, Proposition~1 for a proof that the existence of such a vertex ordering is equivalent to a graph being $d$-degenerate.

\begin{proof}[Proof of Theorem~\ref{Thm: Main Degen}]

    Let $G$ be a $k$-degenerate graph, and let $v_1,...,v_n$ be a linear ordering of $V(G)$ 
    such that for all $i$, $ \vert N(v_i) \cap \{v_{i+1},\dots, v_n\} \vert  \leq k$. 
    We will prove $\alpha_d(G) \geq \frac{(d+1)n}{k+d+1}$ and $\alpha_d(G) \geq n - \alpha_{k-d-1}(G)$
    using separate, but similar, arguments.
    We begin by demonstrating that $\alpha_d(G) \geq \frac{(d+1)n}{k+d+1}$.

    We colour the vertices of $G$ blue or red in order starting from $v_1$, using the following procedure:
    let $B$ be the set of vertices already coloured blue (initially empty), and 
    let $R$ be the set of vertices already coloured red (initially empty).  
    Let the blue-degree of a vertex $v_i$, written $\degb(v_i)$,
    be defined by
    $\degb(v_i) := B \cap N(v_i) \cap \{v_1,...,v_{i-1}\}$, 
    and let the red-degree of a vertex $v_i$, written $\degr(v_i)$,
    be defined by
    $\degr(v_i) := R \cap N(v_i) \cap \{v_{i+1},...,v_n\}$. 
    We colour $v_1$ blue.
    At step $i>1$, we will assume that vertices $v_1,\dots, v_{i-1}$ are already coloured, and that no vertex $v_i,\dots, v_n$
    has been coloured.
    We will then colour vertex $v_i$ blue if $\degb(v_i) \le d$, or red if $\degb(v_i) > d$.

    Note that this procedure will colour all vertices of $G$.
    By our colouring rule, the reader can easily observe that $G[B]$ is a $d$-degenerate graph, and the original ordering (with red coloured vertices removed) witnesses its $d$-degeneracy. 
    Furthermore,
    by our choice of colouring procedure each vertex $v_j \in R$ satisfies $\degb(v_j) \geq d+1$.
    Hence, $\sum_{u \in R} \degb(u) \ge (d+1)\vert R \vert$.
    Next, observe that
    \[
    \sum_{u \in R} \degb(u) = \sum_{v \in B} \degr(v)
    \]
    since both sides of this equation count the number of edges $v_iv_j \in E(G)$
    where $i<j$, $v_i \in B$, and $v_j \in R$.
    Finally, we note that for all vertices $v_i$, $\degr(v_i) \leq  \vert N(v_i)\cap \{v_{i+1}, \dots, v_n\} \vert  \leq k$
    by our choice of vertex order $v_1,\dots, v_n$.
    Thus,  $\sum_{v \in B} \degr(v) \leq k\vert B \vert$. 
    
    We conclude that $(d+1)\vert R \vert \leq k \vert B \vert$.
    This implies $\vert R \vert \leq \frac{k}{d+1} \vert B \vert$.
    Recalling that $B$ and $R$ partition $V(G)$, since all vertices receive a colour, 
    we note that $n = \vert R \vert + \vert B \vert$.
    Hence, 
    \begin{align*}
        n & = \vert R \vert + \vert B \vert \\
          & \leq \Big(\frac{k}{d+1} \vert B \vert \Big) + \vert B \vert \\ 
          & = \frac{k+d+1}{d+1} \vert B \vert.
    \end{align*}
    It follows that $\frac{(d+1)n}{k+d+1} \leq \vert B \vert$.
    Recalling that $G[B]$ is $d$-degenerate this implies $\alpha_d(G) \geq \frac{(d+1)n}{k+d+1}$ as desired.

    Now, we turn our attention to proving $\alpha_d(G) \geq n - \alpha_{k-d-1}(G)$.
    Let $u_1,\dots, u_n$ be a linear ordering of $V(G)$ such that for all $i$, $ \vert N(u_i) \cap \{u_{1},\dots, u_{i-1}\} \vert  \leq k$.
    Such an ordering exists, since one can be obtained by reversing the indices of the order $v_1,\dots, v_n$, that is $u_i = v_{n-i+1}$.

    Preform the same procedure as before to colour vertices of $G$ either blue or red, 
    this time respecting the vertex order $u_1,\dots, u_n$ (rather than $v_1,\dots, v_n$).
    Let $B'$ and $R'$ be the resulting sets of blue and red vertices, respectively.
    Observe that $G[B']$ is $d$-degenerate by the same argument as for $G[B]$.

    We claim that $G[R']$ is $(k-d-1)$-degenerate.
    Notice that if true, $\alpha_d(G) + \alpha_{k-d-1}(G) \geq n$, since $B'$ and $R'$ partition $V(G)$.
    To see that $G[R']$ is $(k-d-1)$-degenerate, we note that, by the definition of our colouring procedure
    each vertex $u_i \in R'$ has $\degb(u_i) \ge d+1$.
    Since, $ \vert N(u_i) \cap \{u_{1},\dots, u_{i-1}\} \vert  \leq k$ by our choice of vertex ordering,
    this implies each vertex $u_i$ has at most $k-d-1$ red neighbours in $\{u_1,\dots, u_{i-1}\}$.
    Hence, $u_1,\dots, u_n$ with the blue vertices removed from it is a $(k-d-1)$-degenerate ordering of $G[R']$.
    This concludes the proof.
\end{proof}

\section{Graphs on Surfaces}

In this section we will prove Theorem~\ref{Thm: Surfaces}.
We begin 
considering graphs on surfaces
by noting a well known 
corollary of Euler's formula.
For completeness a proof is included. 

\begin{lemma}\label{Lemma: Euler Tri-free}
    If $G$ is a triangle-free graph with genus at most $g$, then $ \vert E(G) \vert  \leq 2n+4g-4$.
\end{lemma}

\begin{proof}
    Let $G$ be a graph of Euler genus at most $g$.
    Let $\Pi$ be a minimum genus embedding of $G$.
    By the handshaking lemma, the sum of the size of all faces in $\Pi$ is equal to twice the number of edges in $G$.
    Since $G$ is triangle-free, this implies $F \leq \frac{1}{2}\vert E(G)\vert$, where $F$ denotes the number of faces in $\Pi$.
    Then Euler's formula implies $2-2g \leq n - \frac{1}{2}\vert E(G)\vert$.
    Isolating $\vert E(G)\vert$ implies the desired result.
\end{proof}

Next, we consider another useful
corollary of Euler's formula.
This lemma also
appears in \cite{bradshaw2025injective,clow2024oriented}.
For completeness a proof is included here.
Observe that this implies that if $g$ is a constant,
then
$\alpha_6(G) = (1-o(1))n$
for any graphs $G$ with genus at most $g$.

\begin{lemma}\label{Lemma: Large Min Degree}
    Let $k \geq 1$ be an integer.
    If $G$ is a graph of genus at most $g$ and minimum degree at least $k + 6$, then $G$ has fewer than $\frac{12g}{k}$ vertices.
\end{lemma}
\begin{proof}
    Let $G$ be a graph of Euler genus at most $g$ and $\delta(G) \geq k + 6$.
    By Euler's formula and the handshaking lemma, $n - \frac{1}{3}|E| > -2g$. Rearranging this,
    \[\sum_{v \in V(G)} (\deg(v) - 6) < 12g.\]
    Since, each vertex has degree at least $k + 6$, the number of terms in this sum is less than $\frac{12g}{k}$, completing the proof.
\end{proof}

We are now prepared to prove Theorem~\ref{Thm: Surfaces}.

\begin{proof}[Proof of Theorem~\ref{Thm: Surfaces}]
    Let $G$ be a graph with genus at most $g$.
    The outline of the proof is consistent across different choices of $2 \leq d \leq 5$.
    Unfortunately, the same argument does not apply to the $d=1$ case.
    Hence, 
    we consider these two cases separately.

    Suppose $d\geq 2$.
    Let $v_1,\dots, v_n$ be a fixed but arbitrary linear ordering of $V(G)$.
    Colour vertices blue or red, using the same procedure as in the proof of Theorem~\ref{Thm: Main Degen}
    with respect to the vertex order $v_1,\dots, v_n$ and our choice of $d$.
    Also, let blue-degree and red-degree be defined per the proof of Theorem~\ref{Thm: Main Degen}.
    Let $B$ and $R$ be the resulting set of blue and red vertices, respectively.

    By the definition of our vertex colouring procedure, for all $u \in R$, $\degb(u) \geq d+1$.
    Furthermore, like in the proof of Theorem~\ref{Thm: Main Degen}, $G[B]$ is $d$-degenerate.
    Let $X$ be the set of all edges $v_iv_j$ such that $i<j$, $v_i\in B$, and $v_j \in R$,
    and let $H$ be the subgraph of $G$ consisting of all edges in $X$ and their incident vertices.
    Then $R \subseteq V(H)$, $H$ is a bipartite graph, and
    $X \geq (d+1)  \vert R \vert $.
    
    Since $G$ has genus at most $g$ and $H$ is a subgraph of $G$, $H$ has genus at most $g$.
    Since $H$ is a bipartite graph with genus at most $g$, and $\vert V(H) \vert \leq n$,
    observe that
    Lemma~\ref{Lemma: Euler Tri-free}
    implies that $\vert E(H)\vert = \vert X\vert \leq 2n+4g-4$.

    We have shown that $(d+1)|R|\leq 2n+4g-4$. 
    Hence, $|R| \leq \frac{2n+4g-4}{d+1}$.
    Since $B$ and $R$ partition $V(G)$, 
    $n = \vert B \vert + \vert R \vert$.
    Combining these facts we obtain $\vert B \vert \geq n - \frac{2n+4g-4}{d+1}$.
    Since $G[B]$ is $d$-degenerate,
    and $d\geq 2$
    this completes proof of this case $d \geq 2$.

    Now suppose $d=1$.
    Let $v_1,\dots, v_n$ be a linear ordering of the vertices of $V(G)$
    such that for all $i$, $v_i$ is a vertex of minimum degree in $G[\{v_1,\dots, v_i\}]$.
    Then, Lemma~\ref{Lemma: Large Min Degree} implies for all $i\geq 12g$,
    $|N(v_i) \cap \{v_1,\dots, v_{i-1}\}|\leq 6$.
    Letting $H = G - \{v_1,\dots, v_{12g-1}\}$, this implies $H$ is $6$-degenerate.

    Since $H$ is $6$-degenerate, Theorem~\ref{Thm: Main Degen} implies
    $\alpha_1(H) \geq \frac{2}{7}\vert V(H) \vert$.
    Given $H$ is an induced subgraph of $G$, 
    we note that $\alpha_1(G) \geq \frac{2|V(H)|}{7}.$
    Thus, $\alpha_1(G) \geq \frac{2n-24g+2}{7}$ as required.
    This completes the proof.
\end{proof}

\section{Future Work}

The most natural open problem resulting from our work is to determine exact values for $\alpha_d(k)$. 
After conducting extensive computational work, we conjecture the following.

\begin{conjecture}\label{Conj: 2-degen}
    $\alpha_1(2) = \frac{3}{5}$.
\end{conjecture}

\begin{conjecture}\label{Conj: k>2 and d}
    For all $k\geq 3$ and $k>d$, $\alpha_d(k) = \frac{d+1}{k+1}$.
\end{conjecture}

Our code is available upon request.
As a short summary, using \texttt{gtools}, a package inside of \texttt{nauty} \cite{mckay2014practical}
we exhaustively generate small graphs with degree sequences that were of interest.
When this did not produce examples better than $K_4$ with an edge subdivided, or cliques,
we wrote a genetic algorithm to search for larger graphs that might disprove Conjecture~\ref{Conj: 2-degen}
or Conjecture~\ref{Conj: k>2 and d}. 
Again this did not produce better examples, although in several cases it could find large graphs that match the bounds in 
Conjecture~\ref{Conj: 2-degen}
and Conjecture~\ref{Conj: k>2 and d}.

We computed $\alpha_1(G)$ using an integer programming approach from \cite{melo2022maximum}.
For $d\geq 2$, our only means to calculate $\alpha_d(G)$ was to check the degeneracy of all induced subgraphs.
This computing bottleneck limited our ability to run our genetic algorithm at scale for cases where $d>1$.
Note that there is a Monte Carlo algorithm given in \cite{pilipczuk2012finding} for computing $\alpha_d(G)$.
However, for graphs with a small number of vertices (say $<50$), 
which is where we found the most success for our genetic algorithm, 
we did not find that random algorithms provided a substantial advantage.
Our code was run on the FIR supercomputer at Simon Fraser University.

This leads us to the following problem that might assist future computation work regarding $\alpha_d$.

\begin{problem}
    For each $d\geq 2$, determine an integer program that computes $\alpha_d(G)$ for an input graph $G$,
    or prove no such program exists.
\end{problem}

Given Theorem~\ref{Thm: Surfaces}, it is also natural to ask about induced subgraphs of graphs on surfaces.
We are particularly interested in induced forests in graphs on surfaces.
This leads us to conjecture the following.

\begin{conjecture}\label{Conj: Half forest}
    For all non-negative integers $g$, there exists a constant $f(g)$ such that,
    if $G$ is a graph with genus $g$, then $\alpha_1(G) \geq \frac{n}{2} -f(g)$.
\end{conjecture}

If $G$ is a triangle-free planar graph, then the conjecture is true.
In fact, $\alpha_1(G) \geq \frac{6n+7}{11}$, see \cite{dross2019large}.
While considering if there is an easy proof of 
Conjecture~\ref{Conj: Half forest} for triangle-free graphs,
we stumbled upon the following conjecture.
Note that this conjecture is trivial for graphs of girth $3$.
Observe that if the conjecture is true, then it is best possible,
as witnessed by a disjoint union of cycles of length $k$.

\begin{conjecture}\label{Conj: girth g}
    For all $k\geq 3$, if $G$ is a graph with $m$ edges and girth at least $k$, then $\alpha_1(G) \geq n - \frac{m}{k}$.
\end{conjecture}

\bibliographystyle{abbrv}
\bibliography{bib}

\begin{thebibliography}{10}

\bibitem{akiyama1987maximum}
J.~Akiyama and M.~Watanabe.
\newblock Maximum induced forests of planar graphs.
\newblock {\em Graphs and Combinatorics}, 3(1):201--202, 1987.

\bibitem{albertson1979conjecture}
M.~Albertson and D.~Berman.
\newblock A conjecture on planar graphs. graph theory and related topics (ja bondy and usr murty, eds.), 1979.

\bibitem{albertson1976lower}
M.~O. Albertson.
\newblock A lower bound for the independence number of a planar graph.
\newblock {\em Journal of Combinatorial Theory, Series B}, 20(1):84--93, 1976.

\bibitem{alon1987large}
N.~Alon, J.~Kahn, and P.~D. Seymour.
\newblock Large induced degenerate subgraphs.
\newblock {\em Graphs and Combinatorics}, 3(1):203--211, 1987.

\bibitem{alon1996acyclic}
N.~Alon, B.~Mohar, and D.~P. Sanders.
\newblock On acyclic colorings of graphs on surfaces.
\newblock {\em Israel Journal of Mathematics}, 94(1):273--283, 1996.

\bibitem{borodin1979acyclic}
O.~V. Borodin.
\newblock On acyclic colorings of planar graphs.
\newblock {\em Discrete Mathematics}, 25(3):211--236, 1979.

\bibitem{bradshaw2025injective}
P.~Bradshaw, A.~Clow, and J.~Xu.
\newblock Injective edge colorings of degenerate graphs and the oriented chromatic number.
\newblock {\em European Journal of Combinatorics}, 127:104139, 2025.

\bibitem{caro1979new}
Y.~Caro.
\newblock New results on the independence number.
\newblock Technical report, Technical Report, Tel-Aviv University, 1979.

\bibitem{clow2024oriented}
A.~Clow.
\newblock On oriented colourings of graphs on surfaces.
\newblock {\em arXiv preprint arXiv:2409.13076}, 2024.

\bibitem{cranston2016planar}
D.~W. Cranston and L.~Rabern.
\newblock Planar graphs have independence ratio at least 3/13.
\newblock {\em The Electronic Journal of Combinatorics}, 23(3):P3--45, 2016.

\bibitem{dross2019large}
F.~Dross, M.~Montassier, and A.~Pinlou.
\newblock Large induced forests in planar graphs with girth 4.
\newblock {\em Discrete Applied Mathematics}, 254:96--106, 2019.

\bibitem{dvovrak2018induced}
Z.~Dvo{\v{r}}{\'a}k and T.~Kelly.
\newblock Induced 2-degenerate subgraphs of triangle-free planar graphs.
\newblock {\em The Electronic Journal of Combinatorics}, pages P1--62, 2018.

\bibitem{gu20223}
Y.~Gu, H.~A. Kierstead, S.~Oum, H.~Qi, and X.~Zhu.
\newblock 3-degenerate induced subgraph of a planar graph.
\newblock {\em Journal of Graph Theory}, 99(2):251--277, 2022.

\bibitem{kawarabayashi2010star}
K.~Kawarabayashi and B.~Mohar.
\newblock Star coloring and acyclic coloring of locally planar graphs.
\newblock {\em SIAM Journal on Discrete Mathematics}, 24(1):56--71, 2010.

\bibitem{lick1970k}
D.~R. Lick and A.~T. White.
\newblock k-degenerate graphs.
\newblock {\em Canadian Journal of Mathematics}, 22(5):1082--1096, 1970.

\bibitem{lukot2015maximum}
R.~Lukot'ka, J.~Maz{\'a}k, and X.~Zhu.
\newblock Maximum 4-degenerate subgraph of a planar graph.
\newblock {\em The Electronic Journal of Combinatorics}, pages P1--11, 2015.

\bibitem{mckay2014practical}
B.~D. McKay and A.~Piperno.
\newblock Practical graph isomorphism, ii.
\newblock {\em Journal of symbolic computation}, 60:94--112, 2014.

\bibitem{melo2022maximum}
R.~A. Melo and C.~C. Ribeiro.
\newblock Maximum weighted induced forests and trees: new formulations and a computational comparative review.
\newblock {\em International Transactions in Operational Research}, 29(4):2263--2287, 2022.

\bibitem{MoharThomassen}
B.~Mohar and C.~Thomassen.
\newblock {\em Graphs on surfaces}.
\newblock Johns Hopkins Studies in the Mathematical Sciences. Johns Hopkins University Press, Baltimore, MD, 2001.

\bibitem{pilipczuk2012finding}
M.~Pilipczuk and M.~Pilipczuk.
\newblock Finding a maximum induced degenerate subgraph faster than 2 n.
\newblock In {\em International Symposium on Parameterized and Exact Computation}, pages 3--12. Springer, 2012.

\bibitem{salavatipour2006large}
M.~R. Salavatipour.
\newblock Large induced forests in triangle-free planar graphs.
\newblock {\em Graphs and Combinatorics}, 22(1):113--126, 2006.

\bibitem{wei1981lower}
V.~K. Wei.
\newblock A lower bound on the stability number of a simple graph.
\newblock {\em Bell Laboratories Technical Memorandum Murray Hill, NJ, USA}, 1981.

\bibitem{west2001introduction}
D.~B. West.
\newblock {\em Introduction to Graph Theory}, volume~2.
\newblock Prentice hall Upper Saddle River, 2001.

\end{thebibliography}

\end{document}